\documentclass[10pt,a4paper]{article}

\usepackage{isorot}

\usepackage{graphicx,subfigure}

\usepackage[square,numbers]{natbib}

\usepackage{booktabs}
\usepackage{multirow}

\usepackage{pstricks,pst-plot}

\usepackage{amsmath}
\usepackage{amsthm}
\usepackage{amsfonts}
\usepackage{amssymb}
\usepackage{latexsym}
\usepackage{stackrel}
\usepackage[T1]{fontenc}
\usepackage{IEEEtrantools}
\usepackage{enumerate}
\usepackage{fourier}
\usepackage{skak}

\usepackage{lastpage}

\DeclareMathOperator{\sgn}{sgn}

\let\mathnumsetfont\mathbb
\newcommand\Rset{\mathnumsetfont R} 

\newcommand\M{{\cal M}}

\newtheorem{teo}{Theorem}
\newtheorem{teo*}{Theorem}
\newtheorem*{prop*}{Proposition CK}

\newtheorem{cor}{Corollary}

\def\limsup{\mathop{\overline{\mathrm{lim}}}}
\def\liminf{\mathop{\underline{\mathrm{lim}}}}
\def\limx{\lim_{x\rightarrow\infty}}

\def\limxi{\lim_{\xi\rightarrow\infty}}
\def\limpsi{\lim_{\psi\rightarrow\infty}}

\def\limsupx{\limsup_{x\rightarrow\infty}}
\def\liminfx{\liminf_{x\rightarrow\infty}}

\def\limsuppsi{\limsup_{\psi\rightarrow\infty}}

\title{A note on Tauberian Theorems of Exponential Type}

\author{Meitner Cadena\thanks{UPMC Paris 6 \& CREAR, ESSEC Business School;\, E-mail: meitner.cadena@etu.upmc.fr, b00454799@essec.edu, or meitner.cadena@gmail.com}} 

\date{}

\begin{document}

\maketitle

\begin{abstract}
Tauberian Theorems of exponential type provided by Kohlbecker, de Bruijn, and Kasahara are proved in only one Tauberian theorem.
To this aim, the structure of those classical tauberian theorems is identified and, using a relationship recently proved by Cadena and Kratz, the relationships among its components are given.

\vspace{2ex}


\vspace{2ex}

\end{abstract}

\section{Motivation and main results}

The Tauberian theorems of exponential type given by Kohlbecker, de Bruijn, and Kasahara appeared in 1958, 1959, and 1978, respectively.
They concern equivalences between the logarithm of functions and the logarithm of their Laplace transforms when these two logarithms behave as regularly varying functions.
These theorems are closely related among them and hence their proofs may follow a same structure (see for instance \S 4.12 of \cite{BinghamGoldieTeugels}).
Nevertheless these relationships, these three theorems are often presented independently.
For a survey on these theorems see for instance \cite{BinghamGoldieTeugels}.

We aim to unify these theorems in an only one.
This new presentation gives a general view of these classical results.
As noticed by Bingham et al., a result of this kind was already given by de Bruijn in \cite{deBruijn1959}.
However, our proof is different from that given by this author because the structure of these tauberian theorems is revealed and the interplay among its components is showed.

The Tauberian theorems of exponential type involve regularly varying (RV) functions.
A measurable function $U:\Rset^+\to\Rset^+$ is RV with index $\alpha\in\Rset$ if, for $t>0$, $U(xt)\sim U(x)t^\alpha$ ($x\to\infty$), 
where $f(x)\sim g(x)$ ($x\to x_0$) means $f(x)\big/g(x)\to1$ as $x\to x_0$.
The class of RV functions of index $\alpha$ is denoted by $\textrm{RV}_\alpha$.
If $\alpha=0$, then $U$ is slowly varying (SV).

It follows our main result.

\begin{teo}\label{teo:main:001}~
Let $a,b\in\Rset$ such that $ab(b-1)<0$.
Let $c\in\Rset$ such that $abc<0$.
Let $d:=a(1-b)\big(-ab\big/c\big)^{b/(b-1)}$.
Assume that $P(u)$ is a real function, that $\displaystyle \int_0^rP(u)du$ exists in the Lebesgue sense for every positive $r$, and that $\displaystyle \int_0^\infty P(u)du$ converges if $b<0$.
Put $\displaystyle f(s):=A+\int_0^{\infty}P(us)e^{c u}du$ for some real $A\in\Rset$ such that $A=0$ if $d<0$.
Then
\begin{equation}\label{eq:20150316:001}
\log(P(x))\sim a x^{b}\quad x^b\to\infty
\end{equation}
iff
\begin{equation}\label{eq:20150316:002}
\log(f(\lambda))\sim d\lambda^{b/(1-b)}\quad (\lambda\rightarrow\infty).
\end{equation}
\end{teo}

A relationship provided by Cadena and Kratz \cite{NN2014} is used to prove this result.
For the sake of completeness of this note, we give this relationship as Proposition CK and indicate its proof in appendix.
Part of this proof is copied from \cite{NN2014}.
Our main result is discussed in the last section.

Note that in Theorem \ref{teo:main:001} we use simple forms of RV functions.
They are $\phi\in\textrm{RV}_\alpha$ such that $\phi(x)=x^\alpha$ as $x\to\infty$.
In what follows we use this kind of functions only.
Hence, SV functions are assumed $L(x)=1$ as $x\to\infty$.

It follows the application of our theorem to prove the Tauberian theorems given by Kohlbecker, de Bruijn, and Kasahara.


\begin{cor}[Kohlbecker's Tauberian theorem \cite{Kohlbecker1958}, version given by Bingham et al. \cite{BinghamGoldieTeugels}, pp. 247]
Let $\mu$ be a measure on $\Rset$, supported by $[0;\infty)$ and finite on compact sets.
Let
$$
M(\lambda):=\int_{[0;\infty)}e^{-x/\lambda}d\mu(x)\quad(\lambda>0).
$$
Let $\alpha>1$, $B>0$. 
Then
$$
\log(\mu[0;x])\sim Bx^{1/\alpha}\quad (x\rightarrow\infty)
$$
iff
$$
\log(M(\lambda))\sim (\alpha-1)(B/\alpha)^{\alpha/(\alpha-1)}\lambda^{1/(\alpha-1)}\quad (\lambda\rightarrow\infty).
$$
\end{cor}

\begin{proof}~
By integration by parts $M(\lambda)$ may be rewritten as, using the change of variable $y=x\big/\lambda$, $\displaystyle M(\lambda)=\int_0^\infty e^{-y}\mu\big[0;y\lambda\big]dy$.
Taking $a=B$, $b=\alpha$, and $c=-1$, gives $d=(\alpha-1)(B/\alpha)^{\alpha/(\alpha-1)}$ ($>0$), and putting $P(x)=\mu[0;x]$ and $f=M$ with $A=0$, applying Theorem \ref{teo:main:001}, the corollary then follows.
\end{proof}


As mentioned above, de Bruijn's Tauberian theorem tackled all of three Tauberian theorems of exponential type reviewed in this note.
In order to distinguish the case not concerned in the results of Kohlbecker and Kasahara, in what follows we call this case de Bruijn's Tauberian theorem, as often found in the literature (see for instance \cite{BinghamGoldieTeugels}, \cite{Nane2007}, and \cite{ShanbhagRao}).

\begin{cor}[de Bruijn's Tauberian theorem, \cite{deBruijn1959}, Theorem 2]\label{deBruijn}~
Let $A>0$.
Assume that $P(u)$ is a real function and that $\displaystyle M(\lambda):=\lambda\int_0^{\infty}P(x)e^{-\lambda A x}dx$ converges for all $\lambda>0$.
If $\beta<0$, then for $B<0$,
$$
\log\big(P\big(1\big/x\big)\big)\sim Bx^{-\beta}\quad (x\to\infty)
$$
iff
$$
\log(M(\lambda))\sim B(1-\beta)\left(\frac{\lambda}{B\beta}\right)^{\beta/(\beta-1)}\quad (\lambda\rightarrow\infty).
$$
\end{cor}

\begin{proof}~
Using the changes of variables $y=\lambda x$ and $s=1\big/\lambda$, $\displaystyle M\big(1\big/s\big)=\int_0^\infty e^{-y}P(sy)dy$.
Taking $a=B$, $b=\beta$, and $c=-A$, gives $d=B\big(1-\beta\big)\big(A\big/(B\beta)\big)^{\beta/(\beta-1)}$ ($<0$), and taking $f$ as $f(1\big/\lambda)$ with $A=0$, applying Theorem \ref{teo:main:001}, the corollary then follows.
\end{proof}



\begin{cor}[Kasahara's Tauberian theorem \cite{Kasahara1978}, version given by Bingham et al. \cite{BinghamGoldieTeugels}, pp. 253]\label{teo:KasaharasTauberianTheorem:001}
Suppose $\mu$ be a measure on $(0;\infty)$ such that
$
\displaystyle M(\lambda):=\int_0^\infty e^{\lambda x}d\mu(x)<\infty
$
for all $\lambda>0$.
Let $0<\alpha<1$. 
Then, for $B>0$,
$$
\log \mu\big(x;\infty\big)\sim -Bx^{1/\alpha}\ (<0)\quad (x\rightarrow\infty)
$$
iff
$$
\log(M(\lambda))\sim (1-\alpha)(\alpha/B)^{\alpha/(1-\alpha)}\lambda^{1/(1-\alpha)}\quad (\lambda\rightarrow\infty).
$$
\end{cor}

\begin{proof}~
Noting that $\mu(0;\infty)<\infty$, by integration by parts $M(\lambda)$ may be rewritten as, using the change of variable $y=\lambda x$, $\displaystyle M(\lambda)=\mu(0;\infty)+\int_0^\infty e^{x}\mu\big(x\big/\lambda;\infty\big)dx$.
Taking $a=-B$, $b=1\big/\alpha$, and $c=1$, gives $d=-B\big(1-1\big/\alpha\big)(B/\alpha)^{1/(\alpha-1)}=\big(1-\alpha\big)(B/\alpha)^{\alpha/(\alpha-1)}$, and putting $P(x)=\mu\big(x;\infty\big)$ and $f=M$ with $A=\mu(0;\infty)$, applying Theorem \ref{teo:main:001}, the corollary then follows.
\end{proof}

\section{Proof of Theorem 3.1}
\label{proof}


Assume the hypothesis given in Theorem \ref{teo:main:001}.

Let $0<\epsilon<\big| d \big|\big/2$.
Note that $ d >0$ if $ b >0$, and $ d <0$ if $ b <0$.

\begin{proof}[Proof of the necessary condition]
Define the function $h(x)= a  x^ b + c  x- d $, $x>0$.
$h$ is continuously differentiable, concave ($h''(x)= a  b ( b -1)x^{ b -2}<0$), and, reaches its maximum at $x_M=\big(- c \big/( a  b )\big)^{1/( b -1)}$ ($>0$) and $h(x_M)=0$, so in particular $h\leq0$.
Hence, there exists $0<\eta<\min(x_M,1)$ such that, for $x\in\big[x_M-\eta;x_M+\eta\big]$, $h(x)\geq-\epsilon\big/3$.

Let $0<\tau<1$ be sufficiently small, to be defined later.

Since the function $P$ satisfies \eqref{eq:20150316:001} there exists $x_0>0$ such that, for $x^\beta\geq x_0^\beta$,
\begin{equation}\label{eq:20150316:003}
\displaystyle
\left|\frac{\log(P(x))}{ a  x^{ b }}-1\right|\leq\tau\textrm{.}
\end{equation}
Write, for $\xi>1$ and $\omega\in\big\{\epsilon,-\epsilon\big\}$, using the changes of variable $v=u\big/\log(\xi)$ and $\psi=\log(\xi)$,
\begin{equation}\label{eq:20150316:004}
\frac{f\big((\log\xi)^{(1- b )/ b }\big)}{\xi^{ d +\omega}} =
Ae^{-( d +\omega)\psi}+\psi e^{-\omega\psi}\int_0^{\infty}P\big(v\psi^{1/ b }\big)e^{( c  v- d )\psi}dv\textrm{.}
\end{equation}

If $\omega=-\epsilon$ and $\psi\geq\big(x_0\big/(x_M-\eta)\big)^ b $,
then, denoting $\zeta=-\sgn(a)\tau$ and $\theta=\sgn(b)\eta$, 
provides
$$
e^{-\omega\psi}\int_0^{\infty}P\big(v\psi^{1/ b }\big)e^{( c  v- d )\psi}dv
 \ \geq \ e^{\epsilon\psi}\int_{x_M-\eta}^{x_M+\eta}e^{\left(h(v)+\zeta  a  v^ b \right)\psi}dv
 \ \geq \ 2\eta e^{\frac{2}{3}\epsilon\psi}e^{\zeta a (x_M+\theta)^ b \psi}\textrm{.}
$$
Combining this and \eqref{eq:20150316:004} give, choosing $\tau<\epsilon\big/\big(3 a (x_M+\theta)^ b \big)$ and noting that $\psi\to\infty$ as $\xi\to\infty$,
$$
\displaystyle \limxi\frac{f\big((\log\xi)^{(1- b )/ b }\big)}{\xi^{ d +\omega}}\ \geq\ 
\limpsi\left(Ae^{-( d +\omega)\psi}+2\eta\psi e^{\frac{2}{3}\epsilon\psi}e^{\zeta  a (x_M+\theta)^ b \psi}\right)\ =\ \infty\textrm{.}
$$
Next, take $\omega=\epsilon$. Then, using the changes of variables introduced above,
$$
\int_0^{\infty}P\big(v\psi^{1/ b }\big)e^{ c  v\psi}dv
 \ =\ \int_0^{x_0\psi^{-1/ b }}P\big(v\psi^{1/ b }\big)e^{ c  v\psi}dv
+\int_{x_0\psi^{-1/ b }}^{\infty}P\big(v\psi^{1/ b }\big)e^{ c  v\psi}dv
 \ =\ I_1(\psi)+I_2(\psi)\textrm{.}
$$

On $I_1$, using the change of variable $y=v\psi^{1/ b }$, 
if $ c <0$, then by hypothesis
$$
I_1(\psi)=\psi^{-1/ b }\int_0^{x_0}P(y)e^{ c  y\psi^{1-1/ b }}dy\leq\psi^{-1/ b }\int_0^{x_0}P(y)dy\textrm{,}
$$
and, if $ c >0$, then necessarily $ a >0$ and $ b >1$, and thus
$$
I_1(\psi)=\psi^{-1/ b }\int_0^{x_0}P(y)e^{ c  y\psi^{1-1/ b }}dy\leq 
\psi^{-1/ b }e^{ c  x_0\psi^{\theta}}\int_0^{x_0}P(y)dy\textrm{,}
$$
for some $0<\theta<1$.
So, we get, taking $\psi>(|c|  x_0)^{1/(1-\theta)}$,
$$
\limpsi \psi e^{-(\epsilon+d)\psi}I_1(\psi)\leq
\limpsi \psi^{1-1/ b } e^{-(\epsilon+ d - c  x_0\psi^{\theta-1})\psi}\int_0^{x_0}P(y)dy=0\textrm{.}
$$
On $I_2$, if $ b <0$, $ c <0$ and one has
$$
I_2(\psi) = \psi^{-1/ b } \int_{x_0}^{\infty}P(y)e^{c  y\psi^{-1/ b }\psi}dy
 = \psi^{-1/ b } \int_{x_0}^{\infty}P(y)e^{ c  y\psi^{1-1/ b }}dy\textrm{,}
$$
which implies that, since $e^{ c  y\psi^{1-1/ b }}$ is decreasing in $y$,
$$
\limpsi \psi e^{-(\epsilon+d)\psi}I_2(\psi) \leq \limpsi \psi^{1-1/ b } e^{-(\epsilon+ d )\psi+ c  x_0\psi^{1-1/ b }}\int_{x_0}^{\infty}P(y)dy=0\textrm{.}
$$
If $ b >0$, denote $\zeta$ as above. Then, using \eqref{eq:20150316:003},
$$
e^{-d\psi}I_2(\psi)
\leq \int_{x_0\psi^{-1/ b }}^{\infty}e^{((1-\zeta) a  v^ b + c  v- d )\psi}dv\textrm{.}
$$
Let $g(x)=(1-\zeta) a  x^ b + c  x- d $, $x\geq0$, and take $\displaystyle \zeta<\sgn(1-b)\left(\left[\frac{\epsilon}{2}\left(-\frac{ c }{ a  b }\right)^{1/(1- b )}+1\right]^{1- b }-1\right)$. 
Then, $g$ is differentiable, concave ($g''(x)=(1-\zeta) a  b ( b -1) x^{ b -2}<0$), and reaches its maximum at $x_g=\big(- c \big/( a  b (1-\zeta))\big)^{1/( b -1)}$,
and $\displaystyle g(x_g)=\big(- c \big/( a  b )\big)^{1/( b -1)}\big[(1-\zeta)^{-1/( b -1)}-1\big]$ ($<\epsilon\big/2$).
Hence, $g-\epsilon\big/2<0$.
This inequality and the integrability of $e^{g(x)-\epsilon/2}$ on $(0;\infty)$  
allow again the application of the reverse Fatou lemma giving
$$
\limpsi\int_0^\infty e^{(g(v)-\epsilon/2)\psi}dv\leq\limsuppsi\int_0^\infty e^{(g(v)-\epsilon/2)\psi}dv
\leq\int_0^\infty\limsuppsi e^{(g(v)-\epsilon/2)\psi}dv=0\textrm{.}
$$
Hence, one has
$$
\limpsi \psi e^{-(\epsilon+d)\psi}I_2(\psi)\leq \limpsi \psi e^{-\frac{1}{2}\epsilon\psi}\int_0^\infty e^{(g(v)-\epsilon/2)\psi}dv=0\textrm{.}
$$
Combining the results on $I_1$ and $I_2$ and \eqref{eq:20150316:004} give
$$
\displaystyle \limxi\frac{f\big((\log\xi)^{(1- b )/ b }\big)}{\xi^{ d -\omega}}\ =\ 
\limpsi\left(Ae^{-( d +\omega)\psi}+\psi e^{-(\omega+d)\psi}I_1(\psi)+\psi e^{-(\omega+d)\psi}I_2(\psi)\right)
\ \leq\ 0\textrm{.}
$$

Therefore, $f$ being positive and measurable, $U(\xi)=f\big((\log\xi)^{(1- b )/ b }\big)\in\M$ with $\rho_U= d $, and then, applying Theorem \ref{teo:main:001},
$$
\limxi\frac{\log\left(f\big((\log\xi)^{(1- b )/ b }\big)\right)}{\log(\xi)}= d \textrm{.}
$$
By using the change of variable $\lambda=(\log\xi)^{(1- b )/ b }$ the assertion follows.
\end{proof}

\begin{proof}[Proof of the sufficient condition]
Let $\epsilon>0$.
Suppose that the function $f$ satisfies \eqref{eq:20150316:002}.
Rewriting this limit as, using the change of variable $\xi=\exp\big\{\lambda^{ b /(1- b )}\big\}$,
$$
\limxi\frac{\log\left(f\big((\log\xi)^{(1- b )/ b }\big)\right)}{\log(\xi)}= d \textrm{,}
$$
this means that, applying Theorem \ref{teo:main:001}, $U\in\M$ with $\rho_U= d $ where $U$ is defined as above.
So, one has
$$
\limxi\frac{f\big((\log\xi)^{(1- b )/ b }\big)}{\xi^{ d +\epsilon}}=0
\quad\textrm{and}\quad
\limxi\frac{f\big((\log\xi)^{(1- b )/ b }\big)}{\xi^{ d -\epsilon}}=\infty\textrm{,}
$$
i.e., using the changes of variable $v=u\big/\log(\xi)$ and $\psi=\log(\xi)$ and denoting $Q(x)=\log(P(x))$,
\begin{equation}\label{eq:20150316:005}
\limpsi \psi\int_0^{\infty}e^{(Q(v\psi^{1/ b })/\psi+ c  v- d -\epsilon)\psi}dv=0
\quad\textrm{and}\quad
\limpsi\psi\int_0^{\infty}e^{(Q(v\psi^{1/ b })/\psi+ c  v- d +\epsilon)\psi}dv=\infty\textrm{.}
\end{equation}
We claim that, given $\psi>0$,
\begin{equation}\label{eq:condA}
\textrm{$Q(v\psi^{1/ b })\big/\psi+ c  v- d \leq0$ almost surely (a.s.) for all $v>0$.}
\end{equation}
Assuming there exist $\nu>0$ and $v_1>0$ such that $Q(v_1 \psi ^{1/ b })\big/ \psi + c  v_1- d \geq2\nu$ a.s.,
this means that there exists $\eta>0$ such that, for $v\in\big[v_1-\eta;v_1+\eta\big]$, $Q(v \psi ^{1/ b })\big/ \psi + c  v- d \geq\nu$.
Hence, taking $\epsilon=\nu\big/2$, one gets
$$
\limpsi\psi\int_0^{\infty}e^{(Q(v\psi^{1/ b })/\psi+ c  v- d -\epsilon)\psi}dv\geq
\limpsi\psi\int_{v_1-\eta}^{v_1+\eta}e^{\nu\psi/4}dv=
\limpsi2\eta\psi e^{\nu\psi/4}=\infty\textrm{,}
$$
which contradicts the first limit in \eqref{eq:20150316:005}.

Furthermore, we claim that, given $\psi>0$, 
\begin{equation}\label{eq:condB}
\textrm{there exists $v_0>0$ such that $Q(v_0\psi^{1/ b })\big/\psi+ c  v_0- d =0$.}
\end{equation}
Assuming for all $v>0$ that $Q(v\psi^{1/ b })\big/\psi+ c  v- d <0$, since \eqref{eq:condA} is satisfied,
then, using the change of variable $z=v\psi^{1/ b }$, gives
$$
Q(z)<\frac{ d - c  v}{v^ b }z^ b \textrm{.}
$$
Now, taking the following limits on $v$ provides, for any $z>0$,
$$
Q(z)\leq
\left\{
\begin{array}{ll}
\displaystyle \lim_{v\to0^+}\frac{ d - c  v}{v^ b }z^ b =0 & \textrm{if $ b <0$} \\
 & \\
\displaystyle \lim_{v\to\infty}\frac{ d - c  v}{v^ b }z^ b =-\infty & \textrm{if $0< b <1$, because $ c >0$} \\
 & \\
\displaystyle \lim_{v\to\infty}\frac{ d - c  v}{v^ b }z^ b =0 & \textrm{if $ b >1$.}
\end{array}
\right.
$$
This implies that $P\equiv1$ if $ b <0$ or $b>1$, and $P\equiv0$ if $0< b <1$, which contradicts the hypothesis \eqref{eq:20150316:001}.

Introducing the change of variable $z=v_0\psi^{1/ b }$ in the relationship given in \eqref{eq:condB} gives, for $z>0$,
$$
Q(z)=\frac{ d - c  v_0}{v_0^ b }z^ b \textrm{.}
$$
This implies that $Q$ is continuously differentiable, concave, and then that $Q(v\psi^{1/ b })\big/\psi+ c  v- d$ has a unique maximum at $v$, i.e. $v_0$.
This maximum satisfies
$$
\psi^{1/ b }\frac{Q'(v_0\psi^{1/ b })}{\psi}+ c =
\psi^{1/ b -1}\frac{ d - c  v_0}{v_0^ b } b \left(v_0\psi^{1/ b }\right)^{ b -1}+ c=0 \textrm{,}
$$
which implies $ b ( d - c  v_0)=- c  v_0$, i.e. $v_0= d  b \big/( c ( b -1))$.
$v_0$ is positive and satisfies $v_0=x_M$.
Straightforward computations gives $\displaystyle  a =\big( d - c  v_0\big)\big/v_0^ b $, so $Q$ can be rewritten as $Q(z)= a  z^ b $.
Hence \eqref{eq:20150316:001} follows.
\end{proof}

\section{Discussion of results}

Our proof of the tauberian theorems given by Kohlbecker, de Bruijn, and Kasahara disects the functioning of these theorems.
A function like $h(x)=a  x^ b + c  x- d $, $x>0$, is identified, which has two key properties in order to establish these theorems: concavity and non-positivity.
The first of these properties gives the possible Tauberian theorems: $ab(b-1)<0$, from which exactly three solutions are possible, each one corresponding to a known Tauberian theorem of exponential type.
The second property guarantees the convergence of integrals of type $\displaystyle \int_0^\infty P(us)e^{cu}du$ and lets the control of this integral at $v_0>0$.
This point satisfies $h(v_0)=0$, the unique maximum of $h$.
Note that if $h(v_0)>0$ or $h(v_0)<0$ one cannot obtain those Tauberian theorems.
From the relationship $h'(v_0)=0$ the condition for $c$ is derived, and from $h(v_0)=0$ the corresponding condition for $d$.
Finally, Theorem \ref{teo:main:001} allows the identification of the disposition of the logarithms of functions.

\section*{Acknowledgments} 
The author gratefully acknowledges the support of SWISS LIFE through its ESSEC research program on 'Consequences of the population ageing on the insurances loss'.


\appendix

\section{Proof of Proposition CK}

Let $U:\Rset^+\to\Rset^+$ be a measurable function.

\begin{proof}[Proof of the necessary condition]
Let $\epsilon>0$ and $U\in\M$ with $\rho_U=\tau$.
One has, by definition, that
$$
\limx\frac{U(x)}{x^{\rho+\epsilon}}=0
\quad\textrm{and}\quad
\limx\frac{U(x)}{x^{\rho-\epsilon}}=\infty\textrm{.}
$$
Hence, there exists $x_0\geq1$ such that, for $x\geq x_0$,
$$
U(x)\leq\epsilon x^{\tau+\epsilon}
\quad\textrm{and}\quad
U(x)\geq\frac{1}{\epsilon}x^{\tau-\epsilon}\textrm{.}
$$
Applying the logarithm function to these inequalities and dividing them by $\log(x)$ (with $x>1$) provide
$$
\frac{\log\left(U(x)\right)}{\log(x)}\leq\frac{\log\left(\epsilon\right)}{\log(x)}+\tau+\epsilon
\quad\textrm{and}\quad
\frac{\log\left(U(x)\right)}{\log(x)}\geq-\frac{\log\left(\epsilon\right)}{\log(x)}+\tau-\epsilon\textrm{,}
$$
and, one then has
$$
\limsupx\frac{\log\left(U(x)\right)}{\log(x)}\leq\tau+\epsilon
\quad\textrm{and}\quad
\liminfx\frac{\log\left(U(x)\right)}{\log(x)}\geq\tau-\epsilon\textrm{,}
$$
from which one gets, taking $\epsilon$ arbitrary,
$$
\tau\leq\liminfx\frac{\log\left(U(x)\right)}{\log(x)}\leq\limsupx\frac{\log\left(U(x)\right)}{\log(x)}\leq\tau\textrm{,}
$$
and the assertion follows.
\end{proof}

\begin{proof}[Proof of the sufficient condition]
Let $\epsilon>0$.
By hypothesis, there exists $x_0>1$ such that, for $x\geq x_0$, $\big|\log(U(x))\big/\log(x)-\tau\big|\leq\epsilon\big/2$.

Writing, for $w\in\big\{\epsilon,-\epsilon\big\}$,
$$
\frac{U(x)}{x^{\tau+w}}
=\exp\left\{\log(x)\times\left(\frac{\log(U(x))}{\log(x)}-\tau-w\right)\right\}
$$
gives
$$
\exp\left\{\log(x)\times\left(-\frac{\epsilon}{2}-w\right)\right\}
\leq
\frac{U(x)}{x^{\tau+w}}
\leq
\exp\left\{\log(x)\times\left(\frac{\epsilon}{2}-w\right)\right\}\textrm{,}
$$
and then,
$$
\limx\frac{U(x)}{x^{\tau+\epsilon}}\leq
\limx\exp\left\{\log(x)\times\left(\frac{\epsilon}{2}-\epsilon\right)\right\}=0
$$
and
$$
\limx\frac{U(x)}{x^{\tau-\epsilon}}\geq
\limx\exp\left\{\log(x)\times\left(\frac{\epsilon}{2}+\epsilon\right)\right\}=\infty\textrm{.}
$$
These two limits provide $U\in\M$ with $\rho_U=\tau$.
\end{proof}

\end{document}